\documentclass{amsart}
\usepackage{enumerate}
\usepackage{pb-diagram}
\usepackage{microtype}
\usepackage{amsmath}
\usepackage{amssymb, pb-diagram}
\usepackage[all]{xy}
\usepackage{graphicx} 
\usepackage{mathabx}
\usepackage{chapterbib}
\usepackage{epsfig}
\usepackage{amsfonts}
\usepackage{amssymb}
\usepackage{amsmath}   
\usepackage{amsthm}
\usepackage{color}
\usepackage{latexsym}
\usepackage{hyperref}
\newtheorem{thm}{Theorem}[section]
\newtheorem*{thmi}{Theorem}

\newtheorem{prop}[thm]{Proposition}
\newtheorem{LM}[thm]{Lemma}
\newtheorem*{LMi}{Lemma}
\newtheorem{cor}[thm]{Corollary}
 \theoremstyle{definition}
  \newtheorem{definition}[thm]{Definition}
    \newtheorem{conj}{Conjecture}
 \newtheorem{question}[thm]{Question}
  
  \newtheorem{examples}{Examples}[section]
    \newtheorem{rem}[thm]{Remark}
    
  \DeclareMathOperator {\Hom}{Hom}
   \DeclareMathOperator {\rank}{rank}

   \DeclareMathOperator {\FF}{\mathfrak{F}}

      \DeclareMathOperator {\Ext}{Ext} 
        
 \DeclareMathOperator {\HNN}{HNN}

 \DeclareMathOperator {\res}{res} 
 \DeclareMathOperator {\indu}{ind}

  \DeclareMathOperator {\F}{\mathfrak{F}} 
   
  \DeclareMathOperator {\mor}{mor} 
    \DeclareMathOperator {\Gcd}{Gcd} 
      \DeclareMathOperator {\HF}{{\scriptstyle \mathbf H}\mathfrak F}

    \DeclareMathOperator {\coh}{H} 
 \DeclareMathOperator {\cd}{cd} 
 \DeclareMathOperator {\FA}{FA} 
 \DeclareMathOperator {\Out}{Out} 
   \DeclareMathOperator {\gd}{gd}

  \DeclareMathOperator {\UU}{\mathfrak{U}}  
  \DeclareMathOperator {\BB}{\mathfrak{B}}  
    \DeclareMathOperator {\X}{\mathfrak{X}}  
    \DeclareMathOperator {\Y}{\mathfrak{Y}}

    \DeclareMathOperator {\cat}{CAT(0)} 
        
  \DeclareMathOperator {\pd}{pd} 
   \DeclareMathOperator {\Z}{\mathbb{Z}} 
     
    \DeclareMathOperator {\spli}{spli}
        \DeclareMathOperator {\silp}{silp}

 \DeclareMathOperator {\D}{\Delta} 
 \DeclareMathOperator {\FP}{FP}

    \DeclareMathOperator {\CW}{CW}

 \DeclareMathOperator {\toi}{ \hookrightarrow}    
 \DeclareMathOperator {\tos}{\twoheadrightarrow} 
  
 \DeclareMathOperator{\eg}{\it E_{\FF}{\it G}}

\numberwithin{equation}{section}

\begin{document}

\title[Unbounded torsion and  a conjecture of Kropholler and Mislin]{Some $ {\scriptstyle \mathbf H}_{1}\mathfrak{F}$-groups with unbounded torsion and \\ a conjecture of Kropholler and Mislin}

\author{Giovanni Gandini}
\author{Brita E.A. Nucinkis}

\curraddr{Rheinische Wilhelms-Universit\"{a}t Bonn, Mathematisches Institut, Endenicher Allee 60, 53115 Bonn, GERMANY}
\email{giovanni.gandini@hausdorff-center.uni-bonn.de}
\curraddr{School of Mathematics, University of Southampton, Southampton, SO17 1BJ UNITED KINGDOM}
\email{B.E.A.Nucinkis@soton.ac.uk}

\thanks{The first author was supported by the Leibniz-award of Wolfgang L\"uck and the EPSRC grant EP/J016993/1.}

%    General info

\date{\today}


\begin{abstract} Kropholler and Mislin conjectured that groups acting admissibly on a finite-dimensional $G$-CW-complex with finite stabilisers admit a finite-dimensional model for $\eg,$ the classifying space for proper actions. This conjecture is known to hold for groups with bounded torsion. In this note we consider a large class of groups $\UU$ containing the above and many known examples with unbounded torsion. We show that the conjecture holds for a large subclass of $\UU.$
\end{abstract}

\maketitle

\

\section{Introduction}

Let $G$ be a group and denote by  $\FF$  the class of finite groups. A  $G$-$\CW$-complex  $X$ is  a \emph{classifying space for proper actions of $G$}, or a model for $\eg$, if  the fixed point subcomplex $X^{K}$ is contractible for every finite subgroup $K$ of $G$, and empty otherwise.  
Examples for $\eg$ abound, see for example L\"uck's survey article \cite{luck-05}. 

Generalisations of the constructions of Milnor \cite{milnor} and Segal  \cite{segal} yield a model for $E_{\FF} G$ for every group $G$. However, these constructions lead to infinite models.  The minimal dimension of a model for $\eg$, denoted $\gd_{\FF}G$ is called \emph{the Bredon geometric dimension of $G$}. 
The most natural algebraic counterpart is Bredon cohomology. In Bredon cohomology there is a well-defined notion of  cohomological dimension. For a group $G$, \emph{the Bredon cohomological dimension}, $\cd_{\FF}G$ plays a role analogous to that of  the integral cohomological dimension $\cd G$ in ordinary group cohomology. In particular, $\cd_{\FF} G $ is finite if and only if $\gd_{\FF} G$ is finite \cite{luck-89}. Since both dimensions are often very difficult to compute, various authors  have proposed alternative geometric and algebraic invariants to guarantee the finiteness of $\cd_{\F}G$, and partial results have been proved, see for example \cite{guidosbook, abdeta, nucinkis-00}.

The first of these invariants are $\silp G$ and $\spli G$,  the supremum of the injective lengths of the projective modules and the supremum of the projective lengths of injective modules respectively,  introduced by Gedrich and Gruenberg \cite{GG} in a different context.

In 1993  Kropholler introduced the class  ${\scriptstyle \mathbf H}\mathfrak F$   of \emph{hierarchically decomposable groups} \cite{MR1246274}. The class ${\scriptstyle \mathbf H}\mathfrak F$ is defined as the smallest class of groups containing the class $\FF$ and which contains a group $G$ whenever there is an admissible action of $G$ on a finite-dimensional contractible cell complex for which all isotropy groups already belong to ${\scriptstyle \mathbf H}\mathfrak F$.  The class ${\scriptstyle \mathbf H}\mathfrak F$ can also be defined transfinitely.
Classes of groups with a hierarchical decomposition defined in terms of suitable actions on finite-dimensional complexes appeared previously in the literature, see for examples \cite{alpesha, ikenagafa}. The class of main importance for this note  is the class ${\scriptstyle \mathbf H}_{1}\mathfrak{F}$  \cite{MR1246274}, the first step in the hierarchy.  A group belongs to  ${\scriptstyle \mathbf H}_{1}\mathfrak{F}$   if there is a finite-dimensional contractible $G$-$\CW$-complex $X$ with cell stabilisers in $\mathfrak{F}$.  Clearly every group with finite Bredon geometric dimension lies in ${\scriptstyle \mathbf H}_{1}\mathfrak{F}$, yet it is still unknown whether the converse holds: 

\begin{conj}[Kropholler-Mislin, \cite{guidosbook, MR1851258}] \label{K-M} Every  ${\scriptstyle \mathbf H}_{1}\mathfrak{F}$-group  $G$ admits a finite-dimensional model for $\eg$. 
\end{conj}
 
From a  result proved independently by Bouc \cite{bouc} and Kropholler-Wall \cite{kropwall} it follows that the augmented cellular chain complex $C_{*}(X)$ of a finite-dimensional contractible  $G$-$\CW$-complex with finite stabilisers splits whenever restricted to a finite subgroup of $G$.

Group cohomology relative to a $G$-set $\D$, called $\FF$-cohomology  was introduced    in order to  algebraically  mimic the splitting property of  ${\scriptstyle \mathbf H}_{1}\mathfrak{F}$-groups \cite{nucinkis-99}.  Let $\D$ be a $G$-set satisfying
$$\D^H \neq \emptyset \iff H \in \FF.$$
Relative cohomology with respect to all finite subgroups is  obtained via resolutions of $\Z$ using direct summands of modules of the form $M\otimes \Z\D$, which split when restricted to every finite subgroup. One can now define the $\F$-cohomological dimension of a group, denoted $\F\cd G,$ as the shortest length of such a resolution. In particular, it turns out that $\F\cd G =n$ if and only if there is a resolution by permutation modules with finite stabilisers, which splits when restricted to all finite subgroups. Hence, whenever a group $G$ admits a finite-dimensional model for $\eg$ or belongs to ${\scriptstyle \mathbf H}_{1}\mathfrak{F}$, it has finite $\F\cd G.$ For detail the reader is referred to \cite{nucinkis-99, nucinkis-00}. It is still an open question whether groups of finite $\F$-cohomological dimension belong to  ${\scriptstyle \mathbf H}_{1}\mathfrak{F}.$ 

A related invariant is the Gorenstein dimension of the group. A $\Z G$-module $M$ is said to be Gorenstein projective if it admits a complete resolution in the strong sense, i.e. an acyclic complex of projective modules 
$$ {\bf P_\ast}: \qquad  \cdots \to P_{n+1} \to P_n \to \cdots \to P_0 \to P_{-1} \to P_{-2} \to \cdots$$
where $P_0 \tos M \toi P_{-1},$ and $\Hom_{\Z G}({\bf P_\ast}, Q)$ is acyclic for every $\Z G$-projective module $Q.$ We say a group has finite Gorenstein dimension, $\Gcd G < \infty$, if the trivial module $\Z$ has a finite length resolution by Gorenstein-projectives. This notion goes back to Auslander \cite{auslander} and was recently developed in \cite{abdeta, DT} in the context we are considering here. 

\begin{thm}\cite{GG, abdeta, Emm-10, nucinkis-00}\label{fcd-vs-silp}
For every group $G$ we have 
$$\F\cd G < \infty \implies \silp G = \spli G< \infty \iff \Gcd G < \infty.$$
\end{thm}

L\"uck \cite{lueck-00} defined a further invariant: Let $d \geq 0$ be an integer. We say a group satisfies the property $b(d),$ if whenever a $\Z G$-module $M$ is projective when restricted to a finite subgroup, then $\pd_{\Z G} M \leq d$. A group satisfies $B(d)$ if for every finite subgroup $K$, the Weyl-group $WK = N_G(K)/K$ satisfies $b(d).$ This ties in very closely with the ring $B(G, \Z)$ of bounded functions from $G$ to $ \Z$ having finite projective dimension as a $\Z G$-module. In particular, $B(G, \Z)$ is projective when restricted to a finite subgroup of $G,$ \cite{KT}. Furthermore, 

\begin{prop}\cite[Lemma 7.3]{KM} \cite[Theorem 5.7]{benson} \label{B(d)}
Let $G$ be a group and let $d \geq 0$ be an integer.
\begin{enumerate}
\item $G$ satisfies $B(d)$ implies that $\pd_{\Z G} B(G,\Z) \leq d.$ 
\item For every finite subgroup $K$ of $G$, we have $\pd_{\Z WK} B(WK,\Z) \leq \pd_{\Z G} B(G,\Z).$
\item If $G \in \HF$ then $\pd_{\Z G} B(G,\Z) \leq d$ implies $G$ satisfies $B(d).$ 
\end{enumerate}
\end{prop}

This links with the above invariants as follows:

\begin{prop}\cite[Theorem 4.4]{nucinkis-00} \cite[Theorem C]{CK}\label{link}
\begin{enumerate}
\item $\F\cd G \leq d \implies G$ satisfies $B(d)$
\item If $G\in \HF$, then $\silp G < \infty \implies \pd_{\Z G} B(G,\Z) < \infty.$
\end{enumerate}
\end{prop}

The validity of  Kropholler and Mislin's Conjecture  \ref{K-M} has been verified in several cases, but a general approach is still missing.
In particular, it is known to hold for groups with a bound on the orders of its finite subgroups.

\begin{thm}\cite[Theorem B]{KM}\label{KMthm}%\label{bounded}
Let $G \in \HF$ such that $\pd_{\Z G} B(G,\Z) < \infty$ and that $G$ has a bound on the orders of its finite subgroups. Then $G$ admits a finite-dimensional model for $\eg.$

In particular, let $G\in {\scriptstyle \mathbf H}_{1}\mathfrak{F}$ with a bound on the lengths of its finite subgroups. Then $G$ admits a finite-dimensional model for $\eg.$
\end{thm}

The bound given by Kropholler and Mislin is exponential in $\pd_{\Z G} B(G,\Z).$  L\"uck gives a linear bound in $d$ for groups satisfying $B(d)$ for some $d$. The length of a finite subgroup $H$ of $G$ is the supremum of the lengths of all nested sequences $\{1\} =H_0 < H_1 < H_2 < \cdots H_l=H$ of finite subgroups such that $H_i \neq H_{i+1}$ for all $i=0,\ldots,l-1.$

\begin{thm}\cite[Theorem 1.10]{lueck-00}\label{lueckthm}
Let $G$ be a group satisfying $B(d)$ for some integer $d\geq 0.$ Furthermore, suppose there is a bound $l$ on the lengths of all finite subgroups of $G$. Then $G$ admits a model for $\eg$ of dimension less or equal to $max\{3,d\}+l(d+1).$
\end{thm}

Proposition \ref{link} and Theorem \ref{lueckthm} led the second author to make the following conjecture:

\begin{conj}\cite{nucinkis-00}\label{Nucconj}
Every group of finite $\FF$-cohomological dimension admits a finite-dimensional model for $\eg.$
\end{conj}

By considering invariants related to $\silp G$ and $\spli G$,  see also Theorem \ref{fcd-vs-silp} above, Talelli made the following conjecture:

\begin{conj}\cite{Tal07}\label{Talelliconj}
Every group with $\silp G$ finite admits a finite-dimensional model for $\eg.$
\end{conj}

Theorems \ref{KMthm} and \ref{lueckthm}  imply, that to verify any of the above conjectures it remains to consider groups with unbounded torsion. There are plenty of examples of groups with unbounded torsion satisfying the above conjecture. For example, every countable locally finite group acts on a tree with finite stabilisers. Furthermore, lamplighter groups  are covered by the work of Flores and the second author  on elementary amenable groups \cite{flores-05}. On the other hand, fix a prime $p$, and then  let $\{P_{i}\}_{i\in I}$  be an infinite family of finite $p$-groups of strictly increasing orders, then the group $G=\bigast_{i\in I}P_{i}$ has no bound  on the length of its finite subgroups and is non-amenable. Nonetheless, $G$ acts on a tree with finite stabilisers. In this note we introduce a class of  groups that contains all of the above,  and discuss the validity of Kropholler and Mislin's conjecture as well as those by Talelli and the second author within this class.

\begin{definition} We define  $\mathcal{U}$ to be be the smallest class of groups containing all groups satisfying condition $B(d)$ for some integer $d\geq 0$ and having a bound on the orders of their finite subgroups, which is closed under taking extensions and fundamental groups of graphs of groups. 
\end{definition}

The class $\mathcal U$ contains all groups of finite virtual cohomological dimension, Gromov hyperbolic groups, Burnside groups of large odd exponent, and  more generally,  all groups of finite Bredon cohomological dimension with a bound on the orders of their finite subgroups. Furthermore it contains  all countable locally finite groups, lamplighter groups, Houghton's groups \cite{houghton},  all countable elementary amenable groups, countable free products of finite groups and  Dunwoody's inaccessible group \cite{inacc}. 

We begin by showing  that the class $\mathcal{U}$ admits a natural hierarchical decomposition over the ordinals and establish  some of its basic properties.  We prove that Kropholler and Mislin's conjecture as well as those of Talelli and  Nucinkis  hold within a subclass $\UU^{\star}_{\omega_{0}}$ of $\mathcal{U}$ containing all the above examples.

\begin{thmi} Every  ${\scriptstyle \mathbf H}_{1}\FF$-group contained in the class $\UU^{\star}_{\omega_{0}}$ admits a finite-dimensional classifying space for proper actions.
\end{thmi}

\subsection*{Acknowledgements} The authors thank the referee for their helpful comments and suggestions, and for pointing out   an improvement to our  bound in the Lemma in Remark \ref{bemerkung}.

\section{The class $\UU$ and its hierarchy}

Given a class of groups  $\mathfrak{X}$, define  ${\scriptstyle \mathbf {F_{1}}} \mathfrak{X}$ to be the class consisting of those groups, which are isomorphic to a fundamental group of a graph of groups in $\mathfrak{X}$. Let $\mathfrak{I}$ be the class consisting of the trivial group, then  ${\scriptstyle \mathbf {F_{1}}} \mathfrak{I}$ is the class of free groups.  Furthermore,  if $\mathfrak{X}\subseteq \mathfrak{Y}$ then $\mathfrak{X}\subseteq {\scriptstyle \mathbf {F_{1}}}\mathfrak{X} \subseteq {\scriptstyle \mathbf {F_{1}}}\mathfrak{Y}$.

Let $\X$ and $\Y$ be classes of groups. We denote by $\X\Y$ the class of groups, which are extensions of groups in $\X$ by groups in $\Y$.

\begin{definition} Let $\BB$ denote the class of groups satisfying condition $B(d)$ for some integer $d \geq 0$,  which have a bound on the orders of their finite subgroups. We define:
 
 \begin{itemize}
\item  $\UU_{0}=\mathfrak{J}$
\item
$\UU_{\alpha}=({\scriptstyle \mathbf {F_{1}}}\UU_{\alpha -1})\BB$ if $\alpha$ is a  successor  ordinal, 
\item 
$\UU_{\alpha}=\bigcup_{\beta<\alpha}\UU_{\beta}$ if $\alpha$ is a  limit  ordinal. 
\end{itemize}
The class  $\UU$ is defined as $\UU = \bigcup_{\alpha\geq 0}\UU_{\alpha}$.\
\end{definition}

\begin{rem}\label{U1rem}  ${\mathfrak U}_1=( \mathbf {F_{1}}\mathfrak{I}){\mathfrak B}={\mathfrak B}.$

This follows from the fact that any finite extension of a free group admits a $1$-dimensional model for $\eg$. Now apply \cite[Theorem 3.1]{lueck-00}
to give a finite-dimensional model for $\eg$ for any group in ${\mathfrak U}_1$, implying that it satisfies $B(d')$ for some $d'$. Since $G$ is an extension of a torsion-free group by a group with bounded torsion,  it has a bound on the orders of its finite subgroups. 
\end{rem}

\begin{LM}\label{equiv} The class $\UU$ coincides with the class $\mathcal{U}$.

\begin{proof}
Clearly $\UU \subseteq \mathcal{U}$ and $\UU$ is closed under taking fundamental groups of graphs of groups. In order to show $\mathcal{U} \subseteq \UU$ we need only to verify that the class $\UU$ is extension closed. By Bass-Serre theory it follows that  if $G/N$ acts  on a tree $T$, then $G$ has an action on $T$ such that  $N$ acts trivially. 
Hence,  if $\X$ and $\Y$ are two classes of groups then $\X({\scriptstyle \mathbf {F_{1}}}\Y)\subseteq {\scriptstyle \mathbf {F_{1}}}({\X\Y})$.
 We argue by induction on $\beta$ to show $\UU_{\alpha}\UU_{\beta} \subseteq\UU_{\alpha+\beta}$. If $\beta =1$ then by Lemma \ref{U1rem} $\UU_{\alpha}\UU_{1}=\UU_{\alpha}\BB\subseteq ({\scriptstyle \mathbf {F_{1}}}\UU_{\alpha})\BB=\UU_{\alpha+1}$.
\begin{itemize}
\item Suppose $\beta$ is a successor ordinal, $\beta=\gamma+1$.
\begin{align*} \UU_{\alpha}\UU_{\beta} & =\UU_{\alpha}(({\scriptstyle \mathbf {F_{1}}}\UU_{\gamma})\BB)&& \\
&\subseteq (\UU_{\alpha}({\scriptstyle \mathbf {F_{1}}}\UU_{\gamma}))\BB &&  \text{(by universality, \cite[pg. 2]{fin1})} \\ & \subseteq ({\scriptstyle \mathbf {F_{1}}}(\UU_{\alpha}\UU_{\gamma}))\BB && \text{(by the above)}\\
&  \subseteq ({\scriptstyle \mathbf {F_{1}}}(\UU_{\alpha+\gamma}))\BB && \text{(by induction)}\\
&= \UU_{\alpha+\beta}.&&
 \end{align*}

\item Suppose $\beta$ be a limit ordinal, then 
$\UU_{\beta}=\bigcup_{\gamma < \beta}\UU_{\gamma}$.
\begin{align*} 
\UU_{\alpha}\UU_{\beta}&=\UU_{\alpha}(\bigcup_{\gamma < \beta}\UU_{\gamma})&& \\
&=\bigcup_{\gamma < \beta}\UU_{\alpha}\UU_{\gamma}&& \\
& \subseteq \bigcup_{\gamma < \beta}\UU_{\alpha+\gamma} && \text{(by induction)}\\
& = \UU_{\alpha+\beta} . &&
\end{align*}
\end{itemize}
\end{proof}
\end{LM}

\begin{prop}\label{closu} The class $\UU$ is closed under taking free products with amalgamation, $\HNN$-extensions, countable directed unions, extensions and subgroups. 
\begin{proof} 
It is obvious that $\UU$ is closed under taking free products with amalgamation and HNN-extensions. If $G$ is a countable directed union of groups in $\UU$  then $G$ acts on a tree with stabilisers conjugate to groups in the directed system \cite[Lemma 3.2.3]{hierakro} and so $G\in \UU$.
In Lemma \ref{equiv} it is shown that the class $\UU$ is closed under taking extensions.\\
   For any class of groups $\mathfrak{X}$ we write $G\in {\scriptstyle \mathbf S}\mathfrak{X}$ if $G\leq K\in \mathfrak{X}$. Note that ${\scriptstyle \mathbf {SF_{1}}}\mathfrak{X}\subseteq  {\scriptstyle \mathbf {F_{1}S}} \mathfrak{X}$; in fact if $G\in {\scriptstyle \mathbf {SF_{1}}} \mathfrak{X}$, $G$ is a subgroup of a group $K$ that acts on a tree $T$ with stabilisers in $\mathfrak{X}$ and so $T$ is a $G$-tree with stabilisers that are subgroups of the stabilisers of the $K$-tree $T$. 
Clearly if a class $\mathfrak{X}$ is subgroup closed then ${\scriptstyle \mathbf {SF_{1}}}\mathfrak{X}\subseteq   {\scriptstyle \mathbf {F_{1}}} \mathfrak{X}$ and note that the class $\BB$ is subgroup closed. 
A transfinite induction as in Lemma \ref{equiv} yields that ${\scriptstyle \mathbf {S}}\UU \subseteq\UU$. The main observation here is the following: Let $G \in \UU_\alpha$; then any subgroup of $G$ is an extension of a subgroup of a group in ${\scriptstyle \mathbf {F_{1}}}\UU_{\alpha -1}$ by a subgroup of a group in $\BB$.
\end{proof}
\end{prop}

For a class of groups $\mathfrak{X}$, let ${\scriptstyle \mathbf F}\mathfrak{X}$ be the smallest class of groups containing the class $\mathfrak{X}$, and which contains a group $G$ whenever $G$ can be realised as the fundamental group of a graph of groups already in ${\scriptstyle \mathbf F}\mathfrak{X}$. The class ${\scriptstyle \mathbf F}\FF$ was considered by Richard J. Platten in his PhD thesis. Note that  the classes ${\scriptstyle \mathbf F}\BB$ and ${\scriptstyle \mathbf F}\FF$ differ. Obviously, any  non-trivial group of finite cohomological dimension  with Serre's property $\FA$ does not belong to  ${\scriptstyle \mathbf F}\FF$ but it lies in $\BB$;  examples of such groups appear in \cite{pride}.

\begin{LM} $\UU \subseteq {\scriptstyle \mathbf H}\FF$, ${\scriptstyle \mathbf F}\BB\subsetneq \UU$ and $\UU$ is not closed under taking quotients.

\begin{proof}
$\BB \subset {\scriptstyle \mathbf H}_{1}\FF$.  Any group acting on a tree with stabilisers in ${\scriptstyle \mathbf H}\FF$ obviously lies in ${\scriptstyle \mathbf H}\FF,$   and ${\scriptstyle \mathbf H}\FF$ is extension closed by  \cite[2.3]{MR1246274}. In particular  $\UU \subseteq {\scriptstyle \mathbf H}\FF$.  

Let $H$ be a non-trivial finite group and let $P$ be Pride's group of cohomological dimension equal to two with Serre's property $\FA$ \cite{pride}.  Clearly the group $G=H\wr P$ has no bound on the lengths of its finite subgroups and it lies in $\UU_{2}\backslash\UU_{1}$. By \cite{coukar} the group $G$ has Serre's property $\FA$,  it does not lie in $\BB$ and so $G\notin {\scriptstyle \mathbf F}\BB$. Note that $\cd_{\mathbb{Q}}G\leq 3$. We recall that the rational cohomological dimension of a group $G$,  denoted by $\cd_{\mathbb{Q}}G$, is defined as  the projective dimension over $\mathbb{Q}G$ of the trivial module $\mathbb{Q}$.

Since the class of free groups $\bold{Fr}$ is contained in $\BB$ and $\UU \subseteq {\scriptstyle \mathbf {H}}\FF$  it  is enough to give an example of a finitely generated group that does not belong to ${\scriptstyle \mathbf {H}}\FF$. The first  Grigorchuk group  is a $3$-generator group but it  is does not lie in ${\scriptstyle \mathbf H}\FF$ \cite{gg}. Another classical example of such a group is given by the Thompson group $\bold{F}$ \cite{brge, MR1246274}.
\end{proof}
\end{LM}

 \section{Groups in $\mathfrak{U}$ satisfying $\pd_{\Z G} B(G,\Z) < \infty$}

 We start by recalling some basic facts about Bredon cohomology.  Let $G$ be a group and denote by ${\mathcal O}_{\FF}G$ the orbit category, that is the category having as objects transitive $G$-sets with finite stabilisers and as morphisms the $G$-maps. A Bredon-module is  a contravariant functor $M(-): {\mathcal O}_{\FF}G \to \mathfrak{Ab}.$    The category of Bredon-modules is an abelian category and has enough projectives.  In particular, projectives are direct summands of direct sums of Bredon-modules of the form $\Z[-, G/K]$, where $K \in \FF$ and $[G/H,G/K]$ denotes the set of all morphisms $G/H \to G/K$. Bredon cohomology groups $\coh_{\FF}^{n}(G, -)$ are now computed via projective resolutions of the constant Bredon module $\Z(-)=\underline{\Z}.$ For background the reader is referred to \cite{luck-89}, and for a survey to \cite{MR1851258}.

\begin{thm}\cite[Theorem 0.1]{lueckmeintrupp}\label{lueckdim}
Let $G$ be a group of finite Bredon cohomological dimension. Then $G$ admits a model for $\eg$ of dimension $max\{3, \cd_{\FF} G\}.$
\end{thm} 
 
There now follows a basic lemma in Bredon cohomology, which has also been proved in \cite[Corollary 4.7]{depeta} using spectral sequences.
Let $H$ be a subgroup of $G$. We then have a functor 
$$\begin{array}{lclc} I_H: & \mathcal{O}_{\FF,H} H &\to  & \mathcal{O}_{\FF, G} G \\
& H/K &\mapsto & G/K, \end{array}$$
where ${\FF,H}$ and ${\FF,G}$ denote the families of finite subgroups of $H$ and $G$ respectively.
 
  \begin{LM}\label{MV} Let $T$ be a $G$-tree with edge set $E=\bigsqcup_{i\in I} L_{i}\backslash G $ and vertex set $V=\bigsqcup_{j\in J} N_{j}\backslash G$. Then there is a Mayer-Vietoris sequence:
\begin{multline*}
\dots\to\coh_{\FF}^{n}(G, -)\to \bigoplus_{j\in J}\coh_{\FF}^{n}(N_{j}, \res_{I_{N_{j}}}-)\\ \to \bigoplus_{i\in I}\coh_{\FF}^{n}(L_{i}, \res_{I_{L_{j}}}-)\to \coh_{\FF}^{n+1}(G, -)\to \dots
\end{multline*}
 \begin{proof}
 By Corollary 3.4 in \cite{kropholler-08} the augmented Bredon cell complex 
 $$\Z[-,E]\toi \Z[-, V]\tos\underline{\Z}$$
  is a short exact sequence of $\mathcal{O}G$-modules.
 Now applying the long exact sequence in Bredon cohomology we obtain  
\begin{multline*}
\hspace{.02cm}\dots\to   \Ext_{\FF}^{n}(\underline{\Z}, -)\to \bigoplus_{j\in J}\Ext_{\FF}^{n}(\Z [-, N_{j}\backslash G], -)\\ \to \bigoplus_{i\in I}\Ext_{\FF}^{n}(\Z [-, L_{i}\backslash G], -)\to   
\Ext_{\FF}^{n+1}(\underline{\Z}, -)\to \dots
\end{multline*} 
We show that  $\Ext_{\FF, G}^{n}(\Z [-, H \backslash G], -)\cong \Ext^{n}_{\FF, H} (\underline{\Z},  -)$.\\
 \cite[Lemma 2.7]{symonds-05} implies that  $\Z [-, H \backslash G] \cong \indu_{I_{H}}  \underline{\Z}$. From the adjoint isomorphism it follows that induction along $I_{H}$ is  left adjoint to restriction along $I_{H}$:  $$ \mor_{\FF, G}( \indu_{I_{H}}  \underline{\Z}, -)\\ \cong   \mor_{\mathfrak{F}, H}(\underline{\Z}, \res_{I_{H}}-).$$ Hence the result  follows. 
 \end{proof}
   \end{LM}

  \begin{cor}\label{bound} Let $T$ be a $G$-tree with  vertex set $V=\bigsqcup_{j\in J} N_{j}\backslash G$.  If there is a non-negative integer $n$ such that 
  $\cd_{\FF}N_{i}\leq n$ for all $i$, then $\cd_{\FF}G\leq n+1$.
  \begin{proof} It is an immediate consequence of Lemma \ref{MV}.
  \end{proof}
  \end{cor}

Let $\mathfrak X$ be a class of groups admitting a finite-dimensional model for $\eg,$ and let ${\scriptstyle \mathbf {F_{B}}} \mathfrak{X}$ be the class of groups consisting  of those groups which are isomorphic to a fundamental group of a graph of groups in $\mathfrak{X},$ such that, for all vertex groups  $G_\lambda,$ there is a finite  bound $\bold{B}$ on the differences between  $\pd_{\Z G} B(G_\lambda,\Z)$ and $\cd_{\F}G_\lambda$. Note that for every integer $m\geq 1$ there are examples of groups such that $\pd_{\Z G} B(G,\Z)=2m$ yet $\cd_{\F}G=3m$ \cite{leary-03}.

\begin{definition}
We now define a subclass of $\UU$ using the above closure operation. We define:

\begin{itemize}
\item  $\UU^{\star}_{0}=\mathfrak{I}$
\item
$\UU^{\star}_{\alpha}=({\scriptstyle \mathbf {F_{B}}}\UU^{\star}_{\alpha -1})\BB$ if $\alpha<\omega_{0}$, 
\item 
$\UU^{\star}_{\omega_0}=\bigcup_{\beta<\omega_{0}}\UU^{\star}_{\beta}$. 
\end{itemize}
\end{definition}

\begin{thm}\label{2}Let $G$ be a  group in $\UU^{\star}_{\omega_{0}}$ such that $\pd_{\Z G} B(G,\Z) < \infty.$    Then $G$ admits a finite-dimensional model for $\eg.$

In particular, every  ${\scriptstyle \mathbf H}_{1}\FF$-group contained in the class $\UU^{\star}_{\omega_{0}}$ admits a finite-dimensional model for $\eg.$

\begin{proof} In light of L\"uck's Theorem \ref{lueckdim} it suffices to show that $G$ has finite Bredon cohomological dimension.

Suppose $G\in\UU^{\star}_{\alpha}$ and $\pd_{\Z G} B(G,\Z)=n$. If  $\alpha$ is finite, then $G$ is an extension $N\toi G \tos Q$ with $N\in {\scriptstyle \mathbf {F_{B}}}\UU^{*}_{\alpha-1}$ and $Q\in \mathfrak{B}$. It follows from \cite{CK} that $\pd_{\Z N} B(N,\Z) \leq \pd_{\Z G} B(G,\Z)=n$. Hence every vertex group $N_\lambda$  of the graph of groups associated to $N$ has $\pd_{\Z N_\lambda} B(N_\lambda,\Z)\leq n$, and Bredon cohomological dimension bounded by $n+\mathbf{B}$. By Corollary \ref{bound} we have that  $\cd_{\FF}N\leq n+\bold{B}+1$. 
Since $Q$ is in $\BB$, we have that $\cd_{\FF} Q < \infty$ and that $Q$ has a bound $t$ on the orders of its finite subgroups. An application of \cite[Theorem 3.1]{lueck-00} gives that $\cd_{\FF}G\leq t(n+\bold{B}+1)+\cd_{\FF}Q< \infty$.

Now suppose  $\alpha=\omega_{0}$ and $G\in \UU^{\star}_{\omega_{0}}=(\bigcup_{\beta<\omega_{0}}\UU^{\star}_{\beta})$. Hence $G \in \UU^*_{\beta}$ for some finite $\beta$ and we apply the above.
\end{proof}
\end{thm}

\noindent The examples in  \cite{leary-03} show, that in general $t \neq 1.$

\begin{rem}\label{bemerkung}

One might be tempted to define a slightly larger class of groups by replacing ``bound on the orders of the finite subgroups" by ``bound on the lengths of the finite subgroups" throughout. Let ${\mathfrak B'}$ denote the class of groups satisfying condition $B(d)$ for some integer $d \geq 0$,  which have a bound on the lengths of their finite subgroups.  We then define $\UU'$ as above. This  gives a class equivalent to ${\mathcal U'}$ also defined analogously to the above. But for this class, we cannot use an analogous argument to that of Theorem \ref{2}, as  L\"uck's Theorem \cite[Theorem 3.1]{lueck-00} requires the quotient group to have a bound on the orders of the finite subgroups.  We have the following analogue to Remark \ref{U1rem}:

\begin{LMi}\label{U1lemma} Let $H \toi G \tos Q$ be a group extension, where $H$ is torsion-free of finite cohomological dimension $\cd H =n$  and $Q$ satisfies $B(d)$ for some integer $d \geq 0$ and has a bound $l$ on the lengths of its finite subgroups.   Then
$G$ admits a finite-dimensional model for $\eg$ of dimension $max\{3,l+n+d\}.$

In particular, ${\mathfrak U'}_1=( \mathbf {F_{1}}\mathfrak{I}){\mathfrak B'}={\mathfrak B'}$.

\end{LMi}

\begin{proof} The first assertion is basically a special case of \cite[Corollary 5.2]{MP02}. Every finite subgroup of $G$ is isomorphic to a finite subgroup of $Q$ with length bounded by $l$. Hence, every finite extension $\Gamma$ of $H$ is virtually torsion-free and the lengths of the finite subgroups are bounded by $l$. Hence, by \cite[Theorem 6.4]{lueck-00} $\Gamma$ admits a $max\{3,n\}+l$-dimensional model for $E_{\FF}\Gamma$. By \ref{lueckthm} $Q$ admits  a $max\{3,d\}+l(d+1)$-dimensional model for $E_{\FF}Q$. Now apply  \cite[Corollary 5.2]{MP02} giving   a 
$max\{3,n\}+max\{3,d\}+l(d+2)$-dimensional model for $\eg.$

This bound can be improved to that of the claim as follows: \cite[Theorem 2.7]{abdeta} and \cite[Proposition 2.3]{Tal07} imply that $\pd_{\Z Q}B(Q,\Z)=\Gcd Q=m \leq d.$ Now one can use \cite[Theorem 2.8]{abdeta}, which  states that $\Gcd G \leq \Gcd H+\Gcd Q=\cd H+\Gcd Q,$ hence $\Gcd G \leq n+m.$ Since, by the above, $\cd_{\FF} G <\infty,$ and in particular $\pd_{\Z G} B(G,\Z)<\infty$, a further application of \cite[Theorem 2.7]{abdeta} yields that $\pd_{\Z G} B(G,\Z)=\Gcd G$. We now apply \cite[Theorem 3.10]{mart07}, to get that $\cd_{\FF} G \leq \pd_{\Z G} B(G,\Z) +l \leq n+d+l$ as required.

Now let $G \in {\mathfrak U}_1$. Hence it is an extension of a free group by a group $Q$ in $\mathfrak B$. By the above, $G$ admits a finite-dimensional  model for $\eg$, hence satisfies $B(d')$ for some $d'$, and has a bound on the lengths of its finite subgroups. 
\end{proof}

\end{rem}

\begin{thm}\label{rhotheorem} Suppose that there exists a function $\rho : \mathbb{N} \to \mathbb{N}$ such that $\gd_{\FF}G\leq \rho(\pd_{\Z G} B(G,\Z) )$ for every group $G$ of finite Bredon cohomological dimension. Then every group in $\UU$ such that $\pd_{\Z G} B(G,\Z)< \infty$ admits a finite-dimensional model for $\eg.$

In particular, in this case,  Kropholler and Mislin's conjecture holds inside $\UU$.

\begin{proof} This can be proved similarly to  Theorem \ref{2}, using transfinite induction. Let $G \in \UU$ such that $\pd_{\Z G} B(G,\Z) =n.$ We prove that $\cd_{\FF} G\leq \rho(n)$. The claim is obviously true for $G \in \UU_0$ and $G \in \UU_1$. Now let $G \in \UU_\alpha.$ If $\alpha$ is a successor ordinal, we have, as above that $G$ is an extension $N\toi G \tos Q$ with $N$ a graph of groups $N_\lambda \in \UU_{\alpha-1}$. Now by induction, $\cd_{\FF}N_\lambda \leq \rho(n)$ and hence, as above, $\cd_{\FF}G\leq t(\rho(n)+1)+\cd_{\FF}Q < \infty.$ Hence, by assumption $ \cd_{\FF}G \leq \rho(n).$ 

For $\alpha$ a limit ordinal we have that $G \in \UU_\beta$ for some $\beta < \alpha$ and we are done. 
\end{proof}
\end{thm}

\begin{rem}\label{finiteexamples}  Every group admitting a finite model for $\eg$ lies in $\BB$, as these have finitely many conjugacy classes of finite subgroups \cite{lueck-00}. This includes free groups, Gromov-hyperbolic groups \cite{MS}, $\Out(F_n)$, where $F_n$ is a free group of rank $n$ \cite{vogtman}, mapping class groups \cite{M10} and elementary amenable groups of type $\FP_\infty$ \cite{kropholler-08}.
\end{rem}

\begin{examples}\label{ex} We list further examples of groups in $\UU$. 
\begin{enumerate} 
\item Countable locally free groups belong to $ \UU_{1}$ as they are torsion-free of Bredon-cohomological dimension $\leq 2.$
\item Free Burnside groups $B(m, n)$ of large odd exponent lie in $\UU_{1}\backslash\UU_{0}$.  It is known by \cite{adiansolu} that $B(m, n)$ are infinite for large enough exponent and that they have a bound on the orders of their finite subgroups. By \cite{ivanovburn}  they admit an action on a contractible $2$-dimensional $\CW$-complex with cyclic stabilisers and hence are contained in $\BB \backslash \mathfrak{I}.$ 

\noindent Hence Petrosyan's class $\mathcal{N}^{cell}(\mathcal{P}_{6}) \neq \UU$. If $G\in \mathcal{N}^{cell}(\mathcal{P}_{6}) $, then  either it contains a free subgroup on two generators or it is countable elementary amenable \cite[Theorem 3.9]{neste}.  A finitely generated infinite periodic  group cannot be elementary amenable, therefore free Burnside groups  of large odd exponent are not contained  in $ \mathcal{N}^{cell}(\mathcal{P}_{6})$ but they belong to $\BB$.

\item $\UU_{2}$ contains all countable abelian groups, and $\UU$ contains all countable elementary amenable groups. Every finitely generated abelian group lies in $\BB$ and so every countable  abelian group $G$ can be realised as a group acting on a tree with finitely generated abelian stabilisers and so  $G\in\UU_{2}$. 
Clearly $\BB$ contains all finite groups and $\UU$ is closed under taking countable directed unions and so it contains all countable locally finite groups. By Proposition \ref{closu}, $\UU$ is closed under taking extensions and so it contains all countable elementary amenable groups.

\item  Let $\{F_{i}\}_{i\in I}$ be an infinite countable ordered family of finite subgroups such that $|F_{i}|<|F_{i+1}|$. If $G=\bigast_{i\in I}F_{i}$, then $G\in \UU_{2}\backslash \UU_{1}$. $G$ has no bound on the orders of its finite subgroups but it is realised as the fundamental group of a graph of finite groups, and so $G\in\UU_{2}\backslash\UU_{1}$.

\item For every $n$, Houghton's groups $\mathfrak{H}_{n}$ lie in $ \UU_{2}\backslash \UU_{1}$. The group $\mathfrak{H}_{n}$ is isomorphic to an extension of the infinite countable finitary symmetric group  $\Theta$ by $\mathbb{Z}^{n-1}$. The group $\Theta$ is countable \cite[Exercise 8.1.3]{dixomont} and locally finite.  Hence $\Theta\in {\scriptstyle \mathbf {F_{1}}}\UU_{1}$ and $\mathfrak{H}_{n}\in\UU_{2}$.  
\item
Dunwoody's inaccessible group  $\mathfrak{D}$ lies in $\UU_{3}$ \cite{inacc}.

\noindent The group $\mathfrak{D}$ is the fundamental group of a graph $X$ of groups. Every  edge group is finite and the only non-finite vertex group is isomorphic to a free product with amalgamation $Q_{n}\ast_{H_{\omega}}H$. Where $Q_{n}$ is the fundamental group of an infinite graph of groups with all finite edge and vertex groups, $H_{\omega}$ is an infinite countable locally finite group, hence lies in $ \UU_{2}\backslash\UU_{1}$. $H$ is isomorphic to a semidirect product of the infinite finitary symmetric group on a countable set by  an infinite cyclic group, hence $H\in \UU_{2}\backslash\UU_{1}$. 
It is clear from  the construction  that $\mathfrak{D}\in\UU_{3}$ and $\cd_{\mathbb{Q}}\mathfrak{D}\leq 4$. Note that $\mathfrak{D}$ has no bound on the orders of its finite subgroups, which follows from its construction or from Linnell's theorem on inaccessible groups \cite{linnell}. 
\end{enumerate}
With the exception of elementary amenable groups, it can be seen directly that all the above groups lie in $\UU^*_{\omega_0},$ since for finite groups $K$, $\cd_{\FF} K = \pd_{\Z k} B(K,\Z)=0$ and for finitely generated abelian groups $A$, $\cd_{\FF} A = \pd_{\Z A} B(A,\Z) = h(A)$, the Hirsch-length of $A$. \end{examples}

\begin{LM}
Countable elementary amenable groups lie in $\UU^*_{\omega_0}.$
\end{LM}

\begin{proof}
Suppose first that $G$ is an elementary amenable group with $\pd_{\Z G} B(G,\Z)=d$. \cite[Theorem 2.7]{abdeta} and \cite[Proposition 2.1]{Tal07} imply that  $\spli G \leq d+1.$ It now follows from \cite{GG}, that $h(G) \leq d+1.$  Furthermore, for countable elementary amenable groups, $\cd_{\FF} G \leq h(G)+1$ \cite{flores-05}. Thus $\cd_{\FF}G \leq \pd_{\Z G} B(G,\Z) +2$.
The general case now follows directly using the hierarchical description of elementary amenable groups given in \cite{KLM}: Let $\Y$ denote the class of all finitely generated abelian-by-finite
groups. For each ordinal $\alpha$  the class $\X_\alpha$ is defined by
\begin{eqnarray*}
\X_0 = & \mathfrak{I} &\cr
\X_\alpha = & (L\X_{\alpha -1})\Y & \mbox{if  $\alpha$  is  a 
successor  ordinal,} \cr
\X_\alpha =& \bigcup_{\beta < \alpha} \X_\beta & \mbox{if $ \alpha$
is  a limit  ordinal.} \cr
\end{eqnarray*}
The class of all elementary amenable groups is defined by setting
$$ \X = \bigcup_{\alpha \geq 0} \X_\alpha.$$
Note that countable elementary amenable groups lie in $\X_{\omega_0}$. Let  $\mathfrak{Z}$ be a class of  groups and let $G\in L{\mathfrak{Z}}$ be a countable group, for which there is a finite $\mathbf B$ such that for every finitely generated subgroup $H$ of $G,$ we have $\cd_{\FF}H \leq \pd_{\Z H}B(H,\Z)+\mathbf{B}.$ Then automatically, $G \in {\scriptstyle \mathbf {F_{B}}}{\mathfrak Z}.$ 
Since finitely generated abelian-by-finite groups lie in $\BB$ and by the above, for every countable group in $\X$, there is such a $\mathbf{B}=2$, the claim follows.
\end{proof}

Also note that elementary amenable groups of finite Hirsch-length lie in $\UU^*_2.$ By a Theorem of Hillman and Linnell \cite{HL} they are locally finite-by virtually (torsion-free soluble). Since virtually torsion-free soluble groups of finite Hirsch length belong to $\BB$ and finite groups satisfy $\pd_{\Z G}B(G,\Z)=\cd_{\FF}G=0,$ the above claim follows.

Let   $\underline{G}$ be the finitely generated group   constructed in \cite{dunjo1}  admitting the remarkable decomposition $\underline{G}=A \ast_{\Z}\underline{G}$ . 

\begin{prop}The group $\underline{G}$  has a bound on the orders of its finite subgroups and   belongs to $\BB$. 
\begin{proof} In \cite{dun11} it is shown that the group $\underline{G}$ can be realised as the fundamental group of a graph of groups $Y$ with two orbits of vertices $V Y$ and  two orbits of edges $EY$. Each finite subgroup of $\underline{G}$ must lie in one of the conjugates of the $A$ factors. Since $A = \langle a; b \,|\, b^{3} = 1; a^{-1}ba = b^{-1}\rangle$ every finite subgroup has order  bounded by $3$. Moreover, $\underline{G}\cong G_1*_L G_2$ where $G_1$ and $G_2$ are fundamental groups of graphs of groups with all virtually cyclic stabilisers, and $L$ is locally infinite cyclic. Hence, by Lemma \ref{MV}  and \cite[Theorem 3.1]{lueck-00}  $\underline{G}\in\BB$.
\end{proof}
\end{prop}

This is in contrast with finitely generated infinite groups of the form $G=A \times G$ with $A$ infinite.  By an argument similar to \cite[Theorem 4.11]{gg}, we see that these groups must have infinite rational cohomological dimension. 
Tensoring with $\mathbb{Q}$ over $\Z$  the augmented cellular chain complex  of a classifying space for proper actions  leads to a $\mathbb{Q}G$-projective resolution of the trivial module $\mathbb{Q}$, therefore if $G=A \times G$ (with $A$ infinite)  then $G$ has infinite Bredon cohomological dimension.

\begin{rem} Note that if there is a countable (periodic) ${\scriptstyle \mathbf H}_{1}\FF$-group  that does not lie in $\UU,$ then there exists a finitely generated  (periodic) ${\scriptstyle \mathbf H}_{1}\FF$-group 
with no bound on the orders of its finite subgroups.  To see this, suppose that $G$ is  a countable (periodic) ${\scriptstyle \mathbf H}_{1}\FF$-group not belonging to $\UU$. Then $G$ is the directed union of its finitely generated (periodic) subgroups, which are ${\scriptstyle \mathbf H}_{1}\FF$-groups.  Hence $G$ acts  on a tree with stabilisers conjugate to groups in the directed union. If every stabiliser was in $\UU$ so would be $G$, giving a contradiction. Hence, in particular, at least one of the stabilisers does not lie in $\BB$, hence it can't have a bound on the orders of its finite subgroups.
\end{rem}

  \begin{question}Are there countable ${\scriptstyle \mathbf H}_{1}\FF$-groups not contained in the class $\UU$?
  
The easiest example we know of, which lies in ${\scriptstyle \mathbf H}_{1}\FF$, yet it is unknown whether it belongs to $\UU$, is $SL_3(\mathbb{F}_p[t]).$ Hence it might be advisable to consider the class of $S$-arithmetic groups over global function fields. We start by fixing some notation. Let $S$ be a finite non-empty
set of pairwise inequivalent valuations on  a global function field $K$.  Let $\mathcal{O}_{S}\leq K$ be the
ring of $S$-integers and let $\mathbf{G}$ be a reductive K-group. Given a valuation $v$ of $K$, $K_{v}$ is the completion of $K$
with respect to $v$. If $L/K$ is a field extension, the $L$-$\rank$ of $\mathbf{G}$, $\rank_{L} \mathbf{G}$ is the dimension of a maximal $L$-split torus of $\mathbf{G}$. The $K$-group $\mathbf{G}$ is $L$-isotropic if $\rank_{L}\mathbf{G}\neq 0$. For  any connected non-commutative absolutely almost simple $K$-isotropic $K$-group  $\mathbf{G}$, the $S$-arithmetic subgroup $\mathbf{G}(\mathcal{O}_{S})$ acts on a building $X$  of dimension equal to $k(\mathbf{G}, S) := \sum_{v\in S}\rank_{K_{v}} \mathbf{G}$  \cite{BW}. 
It turns out that the building $X$ is a classifying space for proper actions for $\mathbf{G}(\mathcal{O}_{S})$, see the theorem below. From the homological properties of $\mathbf{G}(\mathcal{O}_{S})$  we are able to easily determine most of    its algebraic and geometric dimensions.

We introduce the notion of \emph{Kropholler dimension}, denoted $\mathfrak{K}(G)$:   given a group $G$, it is the minimal dimension of a contractible  $G$-$\CW$-complex with finite stabilisers. By definition,   $G\in{\scriptstyle \mathbf H}_{1}\mathfrak{F}$ if and only if $\mathfrak{K}(G)<\infty$. 
      \end{question}

\begin{thm}\label{kari}Let  $\bold{H}$ be a connected non-commutative absolutely almost simple $K$-isotropic $K$-group. Then  $\gd_{\FF}\bold{H}(\mathcal{O}_{S})=\mathfrak{K}(\bold{H}(\mathcal{O}_{S}))  =\F\cd \bold{H}(\mathcal{O}_{S})=\cd_{\mathbb{Q}}\bold{H}(\mathcal{O}_{S})=k(\bold{H}, S)$.
\begin{proof} Let $H$ be  $\prod_{v\in S}\bold{H}(K_{v})$, then there is  a $k(\bold{H}, S)$-dimensional Euclidean building $X$ associated to $H$. Since $\bold{H}(\mathcal{O}_{S})$ acts  with finite stabilisers on $X$, and this admits a $\cat$-metric, by \cite[Proposition 5]{bln} we obtain that $\gd_{\FF}\bold{H}(\mathcal{O}_{S}) \leq k(\bold{H}, S)$. By \cite{bgw} the group $\bold{H}(\mathcal{O}_{S})$ is of type $\FP_{ k(\bold{H}, S) -1}$. Note that $\bold{H}(\mathcal{O}_{S})$ has no bound on the orders of its finite subgroups as remarked in \cite{gg2}. If the rational cohomological dimension of $\bold{H}(\mathcal{O}_{S})$ was less than $k(\bold{H}, S)$, an application of \cite[Proposition 1]{bound} would give that $\bold{H}(\mathcal{O}_{S})$ has a bound on the orders of its finite subgroups, a contradiction.
Therefore we have $\cd_{\mathbb{Q}}\bold{H}(\mathcal{O}_{S})=k(\bold{H}, S)$ and the result follows.
\end{proof} 
\end{thm}

 \section{Some consequences}
 
 \noindent Let us begin by recording that Proposition \ref{link} and Theorem \ref{2} imply that both Talelli's Conjecture \ref{Talelliconj} and the second author's Conjecture \ref{Nucconj} hold within $\UU^*_{\omega_0}.$  Furthermore, if we assume that there are functions $\rho_1,\rho_2: \mathbb{N} \to \mathbb{N}$ such that for every group of finite Bredon cohomological dimension we have $\cd_{\FF}G \leq \rho_1(\FF\cd G)$ and $\cd_{\FF}G \leq \rho_2(\silp   G)$ respectively, then arguments analogous to those in Theorem \ref{rhotheorem} imply that both conjectures hold in $\UU.$
    
    \begin{LM}\label{relativemaier} Let $T$ be a $G$-tree with edge set $E=\bigsqcup_{i\in I} L_{i}\backslash G $ and vertex set $V=\bigsqcup_{j\in J} N_{j}\backslash G$. 
Then $\FF\cd(G) \leq \sup\{\FF\cd N_{j} \, |\, j\in J\}+1$. 
\begin{proof}This is an immediate consequence of the Mayer-Vietoris sequence in $\FF$-cohomology associated to the short exact $\FF$-split sequence:
 $\Z E \toi \Z V\tos \Z$ \cite{nucinkis-00}.
\end{proof}
\end{LM}

\begin{cor}\label{relcor} Suppose Conjecture \ref{Nucconj} holds. Then there exists a function $\phi : \mathbb{N} \to \mathbb{N}$ such that  $\gd_{\FF}G\leq \phi(\FF\cd(G))$ for every group $G$ of finite $\FF$-cohomological dimension. 

\begin{proof} Assume by contradiction that  there is no such function. Then there exists an $n>1$ and a family of groups $\{ G\}_{i\in \mathbb{N}}$ such that $\FF\cd G_{i}\leq n$ for every  $i$ and $\lim_{i\to \infty}\gd_{\FF}G_{i} = \infty$ \\
The group $G=\bigast_{i\in \mathbb{N}}G_{i}$ 
  has $\FF\cd G\leq n+1$ but $\gd_{\FF}G=\infty$ a contradiction. The inequality $\FF\cd G\leq n+1$ follows from Lemma \ref{relativemaier}. Since $G$ contains subgroups of arbitrarily large Bredon geometric dimension we have  $\gd_{\FF}G=\infty$.
\end{proof}
\end{cor}

\begin{rem} We can also make a statement in the same manner as in Corollary \ref{relcor}, involving $\silp G= \spli G$ or $\pd_{\Z G}(B(G,\Z))$ and Conjecture \ref{Talelliconj}. The proof is analogous to the above. In particular: Let $\lambda(G)$ denote either $\silp G =\spli G$ or $\pd_{\Z G}(B(G,\Z))$. If Conjecture \ref{Talelliconj} holds, then there is a function $\psi: \mathbb{N} \to \mathbb{N}$ such that $\gd_{\FF}G\leq \psi(\lambda(G))$ for every group $G$ with $\lambda(G) < \infty.$
\end{rem}

\begin{thm}\label{4} Let $T$ be a $G$-tree with edge set $E=\bigsqcup_{i\in I} L_{i}\backslash G $ and vertex set $V=\bigsqcup_{j\in J} N_{j}\backslash G$. 
Then $\mathfrak{K}(G) \leq \sup\{\mathfrak{K}(N_{j})_{j\in J}\}+1$. In particular $G\in {\scriptstyle \mathbf H}_{1}\FF$ if and only if there is a bound on the Kropholler dimensions of the edge and vertex groups.
\begin{proof}Replace the edge and vertex groups with suitable  ${\scriptstyle \mathbf H}_{1}\FF$-spaces of minimal dimension and proceed as in   \cite{luck-05} to obtain an ${\scriptstyle \mathbf H}_{1}\FF$-space for $G$ of dimension equal to $ \sup\{\mathfrak{K}(N_{j})_{j\in J}\}+1$.
\end{proof}
\end{thm}
\begin{cor}If Conjecture \ref{K-M} holds, then there exists a function $\gamma : \mathbb{N} \to \mathbb{N}$ such that $\gd_{\FF}G\leq \gamma(\mathfrak{K}(G))$ for every ${\scriptstyle \mathbf H}_{1}\FF$-group $G$. 
\end{cor}
\begin{proof}  This follows from Theorem \ref{4}, and  can be proved analogously to Corollary \ref{relcor}.
\end{proof}

\begin{question}  Does there exist a function $\gamma : \mathbb{N} \to \mathbb{N}$ such that $\gd_{\FF}G\leq \gamma(\mathfrak{K}(G))?$ for all  groups $G$ admitting a finite-dimensional model for $\eg$.

Can we replace $\mathfrak{K}(G)$ above  by $\FF\cd G, \spli G=\silp G$ or $\pd_{\Z G}B(G,\Z)?$

The only known examples where this function is not the identity are those of \cite{leary-03}, mentioned above.

\end{question}


\begin{thebibliography}{10}
\bibitem{adiansolu}
S.~I. Adian.
\newblock {\em The {B}urnside problem and identities in groups}, volume~95 of
  {\em Ergebnisse der Mathematik und ihrer Grenzgebiete [Results in Mathematics
  and Related Areas]}.
\newblock Springer-Verlag, Berlin, 1979.
\newblock Translated from the Russian by John Lennox and James Wiegold.

\bibitem{alpesha}
R.~C. Alperin and P.~B. Shalen.
\newblock Linear groups of finite cohomological dimension.
\newblock {\em Invent. Math.}, 66(1):89--98, 1982.


\bibitem{auslander}  
M. Auslander
\newblock Anneaux de Gorenstein, et torsion en alg\`ebre commutative.
\newblock S\'eminaire d'Algebra Commutative dirig\`e par Pierre Samuel, 1966/67.


\bibitem{abdeta}
A. Bahlekeh, F. Dembegioti, and O. Talelli.
\newblock Gorenstein dimension and proper actions.
\newblock {\em Bull. Lond. Math. Soc.}, 41(5):859--871, 2009.

\bibitem{benson}
D. Benson
\newblock Complexity and varieties for infinite groups. I 
\newblock J. Algebra 193 (1997), no. 1, 260--287.

\bibitem{bln}
Noel Brady, Ian~J. Leary, and Brita E.~A. Nucinkis, 
\newblock On algebraic and
  geometric dimensions for groups with torsion
  \newblock  J. London Math. Soc. (2) \textbf{64} (2001), no.~2, 489--500.

\bibitem{brge}
K.~S. Brown and R. Geoghegan.
\newblock An infinite-dimensional torsion-free {${\rm FP}_{\infty }$} group.
\newblock {\em Invent. Math.}, 77(2):367--381, 1984.

\bibitem{bgw}
K.-U. {Bux}, R.~{Gramlich}, and S.~{Witzel}.
\newblock {Higher finiteness properties of reductive arithmetic groups in
  positive characteristic: the rank theorem}.
\newblock {\em ArXiv e-prints}, February 2011.

\bibitem{bouc}
S. Bouc.
\newblock Le complexe de cha\^\i nes d'un {$G$}-complexe simplicial acyclique.
\newblock {\em J. Algebra}, 220(2):415--436, 1999.

\bibitem{BW}
K.-U. Bux and K. Wortman.
\newblock Finiteness properties of arithmetic groups over finiteness properties
  of arithmetic groups over function fields.
\newblock {\em Inventiones Mathematicae}, 167:355--378, 2007.


\bibitem{CK}
J. Cornick and P.H. Kropholler
\newblock Homological finiteness conditions for modules over group algebras. 
\newblock J. London Math. Soc. (2) 58 (1998), no. 1, 49Ð62. 

\bibitem{coukar}
Y. Cornulier and A. Kar.
\newblock On property ({FA}) for wreath products.
\newblock {\em J. Group Theory}, 14(1):165--174, 2011.

\bibitem{DT}
F. Dembegioti and O. Talelli
\newblock On a relation between certain cohomological invariants.
\newblock J. Pure Appl. Algebra 212 (2008), no. 6, 1432Ð1437.

\bibitem{dunjo1}
M.~J. Dunwoody and J.~M. Jones.
\newblock A group with strange decomposition properties.
\newblock {\em J. Group Theory}, 1(3):301--305, 1998.

%\bibitem[DKLT02]{DKLT}
%W. Dicks, P.~H. Kropholler, I.~J. Leary, and S. Thomas.
%\newblock Classifying spaces for proper actions of locally finite groups.
%\newblock {\em J. Group Theory}, 5(4):453--480, 2002.

\bibitem{dixomont}
J.~D. Dixon and B. Mortimer.
\newblock {\em Permutation groups}, volume 163 of {\em Graduate Texts in
  Mathematics}.
\newblock Springer-Verlag, New York, 1996.

\bibitem{depeta}
F.~{Dembegioti}, N.~{Petrosyan}, and O.~{Talelli}.
\newblock {Intermediaries in Bredon (Co)homology and Classifying Spaces}.
\newblock {\em ArXiv e-prints}, April 2011.

\bibitem{inacc}
M.~J. Dunwoody.
\newblock An inaccessible group.
\newblock In {\em Geometric group theory, {V}ol.\ 1 ({S}ussex, 1991)}, volume
  181 of {\em London Math. Soc. Lecture Note Ser.}, pages 75--78. Cambridge
  Univ. Press, Cambridge, 1993.

\bibitem{dun11}
M.~J. {Dunwoody}.
\newblock {An Inaccessible Graph}.
\newblock {\em ArXiv e-prints}, June 2010.

\bibitem{Emm-10}
I. Emmanouil.
\newblock On certain cohomological invariants of groups. 
\newblock Adv. Math. 225 (2010), no. 6, 3446Ð3462. 

\bibitem{flores-05}
R.~J. Flores and B. E.~A. Nucinkis.
\newblock On {B}redon homology of elementary amenable groups.
\newblock {\em Proc. Amer. Math. Soc.}, 135(1):5--11 (electronic), 2005.

\bibitem{gg}
G.~{Gandini}.
\newblock {Cohomological invariants and the classifying space for proper
  actions}.
\newblock {\em ArXiv e-prints}, June 2011, to appear Groups, Geom. Dyn.

\bibitem{gg2}
G.~{Gandini}.
\newblock {Bounding the homological finiteness length}.
\newblock {\em ArXiv e-prints}, August 2011, to appear B. Lon. Math. Soc.

\bibitem{GG}
T.V. Gedrich and K.W. Gruenberg.
\newblock Complete cohomological functors on groups.
\newblock Topology Appl. 25 (1987), no. 2, 203Ð223. 

\bibitem{guidosbook}
{\em Guido's book of conjectures}, volume~40 of {\em Monographies de
  L'Enseignement Math\'ematique [Monographs of L'Enseignement Math\'ematique]}.
\newblock L'Enseignement Math\'ematique, Geneva, 2008.
\newblock A gift to Guido Mislin on the occasion of his retirement from ETHZ
  June 2006, Collected by Indira Chatterji.
  
\bibitem{HL}
J.A. Hillman and P.A.  Linnell, P. A.
\newblock Elementary amenable groups of finite Hirsch length are locally-finite by virtually-solvable.
\newblock J. Austral. Math. Soc. Ser. A 52 (1992), no. 2, 237Ð241. 

\bibitem{houghton}
C.~H. Houghton.
\newblock The first cohomology of a group with permutation module coefficients.
\newblock {\em Arch. Math. (Basel)}, 31(3):254--258, 1978/79.

\bibitem{ikenagafa}
B.~M. Ikenaga.
\newblock Homological dimension and {F}arrell cohomology.
\newblock {\em J. Algebra}, 87(2):422--457, 1984.

\bibitem{ivanovburn}
S.~V. Ivanov.
\newblock Relation modules and relation bimodules of groups, semigroups and
  associative algebras.
\newblock {\em Internat. J. Algebra Comput.}, 1(1):89--114, 1991.

\bibitem{MR1246274}
P.~H. Kropholler.
\newblock On groups of type {$({\rm FP})\sb \infty$}.
\newblock {\em J. Pure Appl. Algebra}, 90(1):55--67, 1993.

\bibitem{hierakro}
P.~H. Kropholler.
\newblock Hierarchical decompositions, generalized {T}ate cohomology, and
  groups of type {$({\rm FP})_\infty$}.
\newblock In {\em Combinatorial and geometric group theory ({E}dinburgh,
  1993)}, volume 204 of {\em London Math. Soc. Lecture Note Ser.}, pages
  190--216. Cambridge Univ. Press, Cambridge, 1995.

\bibitem{KLM} P.H.  Kropholler, P.A. Linnell and J.A. Moody.
\newblock Applications of a new $K$-theoretic theorem to soluble groups rings, 
\newblock Proc. Amer. Math. Soc. {\bf 104} (3), (1988), 675--684. 

\bibitem{kropholler-08}
P.~H. Kropholler, C. Mart{\'{\i}}nez-P{\'e}rez, and B. E.~A.
  Nucinkis.
\newblock Cohomological finiteness conditions for elementary amenable groups.
\newblock J. Reine Angew. Math. 637 (2009), 49Ð62. 

\bibitem{KM}
P.H. Kropholler and G. Mislin.
\newblock Groups acting on finite-dimensional spaces with finite stabilizers.
\newblock Comment. Math. Helv. 73 (1998), no. 1, 122Ð136. 


\bibitem{KT} P.H. Kropholler and O. Talelli.
\newblock On a property of fundamental groups of graphs of finite groups.
\newblock J. Pure Appl. Algebra 74 (1991), no. 1, 57--59. 

\bibitem{kropwall}
P.~H. Kropholler and C.T.C. Wall.
\newblock Finite group actions and contractible cell complexes.
\newblock {\em Publicacions Matem{\`a}tiques}, 55(1):3--18, 2011.


\bibitem{linnell}
P.~A. Linnell.
\newblock On accessibility of groups.
\newblock {\em J. Pure Appl. Algebra}, 30(1):39--46, 1983.

\bibitem{bound}
I.~J. Leary and B. E.~A. Nucinkis.
\newblock Bounding the orders of finite subgroups.
\newblock {\em Publ. Mat.}, 45(1):259--264, 2001.

\bibitem{leary-03}
I.~J. Leary and B. E.~A. Nucinkis.
\newblock Some groups of type {$VF$}.
\newblock {\em Invent. Math.}, 151(1):135--165, 2003.

\bibitem{luck-89}
W. L{\"u}ck.
\newblock {\em Transformation groups and algebraic {$K$}-theory}, volume 1408
  of {\em Lecture Notes in Mathematics}.
\newblock Springer-Verlag, Berlin, 1989.
\newblock Mathematica Gottingensis.

\bibitem{lueck-00}
W. L{\"u}ck.
\newblock The type of the classifying space for a family of subgroups.
\newblock {\em J. Pure Appl. Algebra}, 149(2):177--203, 2000.

\bibitem{luck-05}
W. L{\"u}ck.
\newblock Survey on classifying spaces for families of subgroups.
\newblock In {\em Infinite groups: geometric, combinatorial and dynamical
  aspects}, volume 248 of {\em Progr. Math.}, pages 269--322. Birkh\"auser,
  Basel, 2005.
  
 \bibitem{lueckmeintrupp}
W. L\"uck and D. Meintrup.
\newblock On the universal space for group actions with compact isotropy.
\newblock  Geometry and topology: Aarhus (1998), 293--305,
Contemp. Math., 258, Amer. Math. Soc., Providence, RI, 2000.  
  
\bibitem{MP02}
C. Mart\'inez-P\'erez.
\newblock A spectral sequence in Bredon (co)homology. 
\newblock J. Pure Appl. Algebra 176 (2002), no. 2-3, 161--173. 

\bibitem{mart07}
C. Mart\'inez-P\'erez.
\newblock A bound for the Bredon cohomological dimension. 
\newblock J. Group Theory 10 (2007), No. 6, 731--747.

\bibitem{milnor}
J. Milnor.
\newblock Construction of universal bundles.
\newblock {\em Ann. of Math. (2)}, 63:272--284, 1956.

\bibitem{MR1851258}
G. Mislin.
\newblock On the classifying space for proper actions.
\newblock In {\em Cohomological methods in homotopy theory ({B}ellaterra,
  1998)}, volume 196 of {\em Progr. Math.}, pages 263--269. Birkh\"auser,
  Basel, 2001.
  
  \bibitem{M10}
  G. Mislin.
\newblock Classifying spaces for proper actions of mapping class groups. 
\newblock M\"unster J. Math. 3 (2010), 263--272. 



\bibitem{MS}
D. Meintrup and T. Schick.
\newblock A model for the universal space for proper actions of a hyperbolic
  group.
\newblock {\em New York J. Math.}, 8:1--7 (electronic), 2002.

\bibitem{nucinkis-99}
B. E.~A. Nucinkis.
\newblock Cohomology relative to a {$G$}-set and finiteness conditions.
\newblock {\em Topology Appl.}, 92(2):153--171, 1999.

\bibitem{nucinkis-00}
B. E.~A. Nucinkis.
\newblock Is there an easy algebraic characterisation of universal proper
  {$G$}-spaces?
\newblock {\em Manuscripta Math.}, 102(3):335--345, 2000.

\bibitem{neste}
N.~{Petrosyan}.
\newblock {New Action-Induced Nested Classes of Groups and Jump (Co)homology}.
\newblock {\em ArXiv e-prints}, November 2009.

\bibitem{pride}
S.~J. Pride.
\newblock Some finitely presented groups of cohomological dimension two with
  property ({FA}).
\newblock {\em J. Pure Appl. Algebra}, 29(2):167--168, 1983.

\bibitem{fin1}
D.J.~S. Robinson.
\newblock {\em Finiteness conditions and generalized soluble groups. {P}art 1}.
\newblock Springer-Verlag, New York, 1972.
\newblock Ergebnisse der Mathematik und ihrer Grenzgebiete, Band 62.

\bibitem{segal}
G. Segal.
\newblock Classifying spaces and spectral sequences.
\newblock {\em Inst. Hautes \'Etudes Sci. Publ. Math.}, (34):105--112, 1968.

\bibitem{symonds-05}
P. Symonds.
\newblock The {B}redon cohomology of subgroup complexes.
\newblock {\em J. Pure Appl. Algebra}, 199(1-3):261--298, 2005.

\bibitem{Tal07}
O. Talelli.
\newblock On groups of type $\Phi$.
\newblock Arch. Math. (Basel) 89 (2007), no. 1, 24Ð32. 

%\bibitem[Tal11]{carata}
%Olympia Talelli.
%\newblock A characterization of cohomological dimension for a big class of
 % groups.
%\newblock {\em J. Algebra}, 326:238--244, 2011.

\bibitem{vogtman}
K. Vogtmann.
\newblock Automorphisms of free groups and outer space. 
\newblock Proceedings of the Conference on Geometric and Combinatorial Group Theory, Part I (Haifa, 2000).
Geom. Dedicata 94 (2002), 1Ð31. 

\end{thebibliography}
\end{document}